\documentclass[11pt]{article}
\usepackage{amssymb,amsmath,amsfonts}
\usepackage{amsthm}
\usepackage{setspace}

\usepackage{mathabx}

\numberwithin{equation}{section} \theoremstyle{plain}
\newtheorem{theorem}{Theorem}[section]
\newtheorem{corollary}[theorem]{Corollary}

\newtheorem{lemma}[theorem]{Lemma}

\theoremstyle{definition}
\newtheorem{definition}{Definition}[section]
\newtheorem{example}{Example}[section]
\theoremstyle{remark}

\allowdisplaybreaks

\begin{document}

\title{$q$-Stieltjes classes for some families of $q$-densities }
\author{Sofiya Ostrovska and Mehmet Turan}
\date{}
\maketitle

\begin{center}
{\it Atilim University, Department of Mathematics,  Incek  06836, Ankara, Turkey}\\
{\it e-mail: sofia.ostrovska@atilim.edu.tr, mehmet.turan@atilim.edu.tr}\\
{\it Tel: +90 312 586 8211,  Fax: +90 312 586 8091}
\end{center}

\begin{abstract}
The Stieltjes classes play a significant  role  in the moment problem allowing to exhibit explicitly infinite families of probability densities with the same sequence of moments. 
In this paper, the notion of $q$-moment determinacy/indeterminacy is proposed and some conditions for a distribution to be either $q$-moment determinate or indeterminate in terms of its $q$-density have been obtained. Also, a $q$-analogue of Stieltjes classes is defined for $q$-distributions and $q$-Stieltjes classes have been constructed for a family of  $q$-densities of $q$-moment indeterminate distributions.

\end{abstract}



\section{Introduction }

While Stieltjes classes de facto appeared in \cite{stieltjes}, the name itself is quite recent.
For good reasons, J. Stoyanov \cite{jap} suggested to use the name `Stieltjes classes'
and  launched their study as a systematic research area. Recent developments
showed that these classes are instrumental in the moment problem in general,
and for probability distributions, in particular \cite{recent, pakes,jordanlenya}.

In the present article, the notion of Stieltjes classes will be adopted with regard to $q$-distributions. These distributions are coming from $q$-calculus and are widely used in applications. See, for example, \cite{char, kemp} and the references therein.
For terminology and basic facts on $q$-distributions we refer to \cite{char}.
In this paper, it is always assumed that $0<q<1.$ The following definitions are employed in the sequel.

\begin{definition} $\cite{char}$
A function $f(t),$ $t>0,$ is a {\it $q$-density} of a random variable $X$ if the distribution function of $X$ is
\begin{align} \label{F1}
F_X(x)=F(x)=\int_0^x f(t)d_qt, \quad x>0.
\end{align}
\end{definition}

The $q$-integral which appears in the definition was proposed by Jackson as below:
$$
\int_0^{x} f(t) d_qt=x(1-q)\sum_{j=0}^\infty f(xq^j)q^j, \quad
\int_0^{\infty} f(t) d_qt=(1-q)\sum_{j=-\infty}^\infty f(q^j)q^j.
$$
See \cite[Sec. 1.11]{gasper}.
If $F$ is continuous at $0,$ then $f$ can be written as the {\it $q$-derivative} of $F:$
$$
f(x)=D_q F(x):=\frac{F(x)-F(qx)}{(1-q)x}, \quad x\neq 0.
$$

\begin{definition} $\cite{char}$
Given $q$-density $f$ of a random variable $X,$ the $k$-th order $q$-moment of $X$ is defined by
\begin{align} \label{mom}
m_q(k;X):=\int_0^\infty t^k f(t)d_qt, \quad k\in{\mathbb N}_0.
\end{align}
\end{definition}

It is evident that the magnitudes of $q$-moments depend only on the values taken on by a $q$-density on the sequence
$
{\mathcal I}_q:=\{q^j\}_{j\in{\mathbb Z}}.
$
It is natural, therefore, to consider the following equivalence relation for functions on $(0, \infty):$
\begin{align} \label{equi}
f\sim g \ \Leftrightarrow \ f(q^j)=g(q^j), \quad j\in{\mathbb Z}.
\end{align}
In other words, functions $f$ and $g$ are equivalent if they coincide on ${\mathcal I}_q.$
If $X$ has finite $q$-moments of all orders, then the probability distribution of $X$ can be classified either as
$q$-moment determinate or $q$-moment indeterminate. More precisely, probability distribution $P_X$ is {\it $q$-moment determinate} if $m_q(k;X)=m_q(k;Y)$ for all $k\in{\mathbb N}_0$ implies that   $f_X\sim f_Y$.   Otherwise, $P_X$ is {\it $q$-moment indeterminate}. In the latter case, $q$-Stieltjes classes for $f$ provide infinite families of not equivalent
$q$-densities with the same $q$-moments as $P_X.$ To be specific, the following $q$-analogues of the notions put forth by J. Stoyanov \cite{jap} are proposed:
\begin{definition}
Let $f(t), t>0$ be a $q$-density of a random variable $X.$ A function $h(t),$ $t>0$ is a {\it $q$-perturbation} for $f$ if
$M_h:=\displaystyle\sup_{t\in{\mathcal I}_q} |h(t)|=1$ and
\begin{align*}
\int_0^\infty  t^k f(t) h(t)  d_qt=0 \quad  \text{for all } k\in{\mathbb N}_0.
\end{align*}
\end{definition}

\begin{definition}
Let $f(t)$ be a $q$-density and $h(t)$ be a $q$-perturbation for $f.$ The set
\begin{align*}
{\bf S}:=\{g: g \text{ is a } q\text{-density and } \: g\sim (1+\varepsilon h)f, \ \varepsilon \in[-1, 1]\}
\end{align*}
is called a {\it $q$-Stieltjes class} for $f$ generated by $h.$
\end{definition}
It has to be pointed out that, in general, $(1+\varepsilon h)f$ is not a $q$-density. However, as it will be shown in Lemma \ref{lem1}, there exists a family of $q$-densities equivalent to $(1+\varepsilon h)f$ for all $\varepsilon\in[-1, 1]$ in the sense of \eqref{equi}. Differently put, given $f$ and $h,$ a $q$-Stieltjes class consists of all $q$-densities $g(t)$ satisfying $g(q^j)=f(q^j)[1+\varepsilon h(q^j)],$ $\varepsilon[-1, 1],$ whenever $j\in{\mathbb Z}.$
Obviously, a $q$-perturbation function and a $q$-Stieltjes class exist only for $q$-indeterminate distributions.
In this work, $q$-Stieltjes classes are constructed for a collection of $q$-densities $f$ which satisfy the next estimate for some positive constant $C:$
$f(q^{-j})\geqslant C q^{j(j+1)/2},$ $j\geq 0.$
That is, the $q$-density $f$ has rather heavy tail.
The sharpness of this result is demonstrated by Theorem \ref{th2}, where it is proved that if a $q$-density $f$ satisfies the condition
$f(q^{-j})=o(q^{j(j+1)/2})$ as $j\to \infty,$
then the distribution $P_X$ is $q$-moment determinate.

Recall the two $q$-analogues of the exponential function:
\begin{align*}
e_q(t)=\prod_{j=0}^\infty \left(1-t(1-q)q^j\right)^{-1}
\end{align*}
and
\begin{align*}
E_q(t)=\prod_{j=0}^\infty \left(1+t(1-q)q^j\right).
\end{align*}
See \cite[formula (1.24)]{char} and \cite[Sec 1.3]{gasper}. Note that $e_q(-t)E_q(t)=1.$
The distribution whose $q$-density equals $\lambda e_q(-\lambda t), \lambda>0, t>0$ is called a $q$-exponential distribution (of the first kind). See \cite[Corollary 2.1]{char}.
It should be emphasized that the $q$-moment (in)determinacy of the $q$-exponential distribution depends on the value of parameter $\lambda,$ in contrast to the well-known classical exponential distribution which is moment determinate for all $\lambda.$ See Examples \ref{ex1} and \ref{ex2}.

For the sequel, we need the following identity attributed to Euler:
\begin{align}\label{eul}
E_q\left(\frac{t}{1-q}\right)=\prod_{j=0}^\infty (1+q^j t) =
\sum_{j=0}^\infty \frac{q^{j(j-1)/2}}{(q;q)_j}\, t^j,
\end{align}
where $(a;q)_j$ is the $q$-shifted factorial defined by:
$$(a;q)_0:=1,\quad (a;q)_j:=\prod_{s=0}^{j-1}(1-aq^s), \quad a\in{\mathbb C}.$$
It is known \cite[formula (2.6)]{zeng} that for some positive constants $C_1,$ $C_2$ and $t$ large enough,
\begin{align}\label{zeng}
C_1 \exp\left\{\frac{\ln^2 t}{2\ln(1/q)}+\frac{\ln t}{2}\right\} \leqslant
E_q\left(\frac{t}{1-q} \right) \leqslant
C_2 \exp\left\{\frac{\ln^2 t}{2\ln(1/q)}+\frac{\ln t}{2}\right\}.
\end{align}
Throughout the paper, the letter $C$ with or without an index denotes a positive constant whose exact value does not have to be specified. Also, the notation 
\begin{align*}
M(r;f):=\max_{|z|=r} |f(z)|
\end{align*}
commonly adopted in the theory of analytic functions will be used repeatedly.

\section{Statement of Results}
To begin with, notice that, while all $q$-densities are non-negative on $(0, \infty)$ and normalized by
\begin{align}\label{one}
\int_0^\infty f(t)d_q t =1,
\end{align}
these two conditions do not guarantee that $f$ is a $q$-density, in distinction from probability densities. However, as the next lemma shows, a non-negative function $f$ satisfying \eqref{one} is equivalent to a $q$-density. What is more,
each equivalence class of a $q$-density $f$ contains infinitely many $q$-densities.
\begin{lemma}\label{existence}
Let $g(t)\geqslant 0,$ $t>0$ and $\int_0^\infty g(t)d_qt=1.$ Then, there exists a $q$-density $f$ such that $f\sim g.$
\end{lemma}
\begin{proof} Clearly, by \eqref{F1}, we have to find a distribution function $F(x)$ so that $D_qF \sim g,$ that is,
$D_q F(q^j)=g(q^j)$ for all $j\in {\mathbb Z}.$ Given $g(t),$ set $F(x)=0$ for $x\leqslant 0,$
$$
F(q^j)=(1-q)\sum_{\ell=j}^\infty g(q^\ell) q^\ell \quad \text{if } \ x=q^j, \ j\in{\mathbb Z},
$$
and define $F$ on each $(q^{j+1}, q^{j})$ in such a way that $F(x)$ is non-decreasing on ${\mathbb R.}$
Now,
\begin{align*}
\lim_{x\to \infty} F(x) = \lim_{j\to-\infty} F(q^j) 
=(1-q)\sum_{\ell=-\infty}^\infty g(q^\ell) q^\ell
=\int_0^\infty g(t)d_q t = 1.
\end{align*}
Therefore, $F(x)$ is a distribution function.
Clearly, $D_q F(q^j)=g(q^j)$ for all $j\in{\mathbb Z},$ that is, $D_q F \sim g$ as desired.
\end{proof}

The next theorem provides a criterion for $q$-densities to be $q$-moment indeterminate.
Furthermore, the proof reveals a $q$-perturbation function for such $q$-densities, which permits to present explicitly a $q$-Stieltjes class.

\begin{theorem}\label{th1}
Let $f(t)$ be a $q$-density of a random variable $X$ possessing finite $q$-moments of all orders. If there is a positive constant $C$ such that
\begin{align} \label{W}
f(q^{-j})\geqslant C q^{j(j+1)/2} \quad for\;\;all\;\; j\geqslant 0,
\end{align}
then the distribution of $X$ is $q$-moment indeterminate.
\end{theorem}

\begin{proof}
To prove the theorem, it suffices to find a $q$-perturbation of $f.$ Let $\tilde{h}(t),$  $t\in(0,\infty)$ be a function such that
\begin{align}\label{h1}
\tilde{h}(q^{-j})=
\begin{cases} (-1)^j\frac{q^{j(j+1)/2}}{(q;q)_j f(q^{-j})}, & j=0,1,2,\ldots\\
0, & j=-1,-2,\ldots
\end{cases}
\end{align}
Clearly, by \eqref{W}, $\tilde{h}\not\sim 0$ is bounded on ${\mathcal I}_q.$
Consider $\varphi(t)=\prod_{s=1}^\infty (1-q^s t).$ With the help of Euler's identity \eqref{eul}, one has:
$$
\varphi(t)=\sum_{j=0}^\infty (-1)^j \frac{q^{j(j+1)/2}}{(q;q)_j}\, t^j.
$$
Evidently, $\varphi(q^{-m})=0$ for all $m=1,2,\ldots,$ or $\varphi(q^{-(k+1)})=0$ for all $k\in{\mathbb N}_0.$
That is,
\begin{align*}
\sum_{j=0}^{\infty}(-1)^j \frac{q^{j(j+1)/2}}{(q;q)_j}\, q^{-j(k+1)} = 0 \quad \text{for all } \: k\in{\mathbb N}_0,
\end{align*}
which implies that
\begin{align*}
\sum_{j=-\infty}^{\infty} f(q^{-j}) \tilde{h}(q^{-j}) q^{-j(k+1)} = 0 \quad \text{for all } \: k\in{\mathbb N}_0,
\end{align*}
or, equivalently,
\begin{align*}
\int_0^\infty t^k f(t) \tilde{h}(t) d_q t = 0 \quad \text{for all } \: k\in{\mathbb N}_0.
\end{align*}
Thus, $h(t)=\tilde{h}(t)/M_{\tilde{h}}$ is a $q$-perturbation of $f,$ and the proof is complete.
\end{proof}

\begin{corollary}\label{cor1}
Let $f$ satisfy \eqref{W}, and construct $\tilde{h}$ as in \eqref{h1}. Then $h(t)=\tilde{h}(t)/M_{\tilde{h}}$ is a $q$-perturbation of $f$ and the set
\begin{align*}
{\bf S}=\left\{g: g \text{ is a } q\text{-density and } g\sim (1+\varepsilon h)f, \ \varepsilon\in[-1, 1]\right\}
\end{align*}
is a $q$-Stieltjes class for $f.$
\end{corollary}

The next example demonstrates an application of this result to $q$-exponential distribution whose $q$-density is given by $f(t)=\lambda e_q(-\lambda t).$ It will be shown that the $q$-moment (in)determinacy of this distributions depends on $\lambda.$

\begin{example} \label{ex1} Let $f(t)=\lambda e_q(-\lambda t)$ be the $q$-density of the $q$-exponential distribution with parameter $\lambda.$ Then
\begin{align*}
f(q^{-j})&=\lambda e_q(-\lambda q^{-j})
= \lambda \prod_{s=0}^{\infty} \left[1+\lambda(1-q)q^{s-j}\right]^{-1}
=\lambda e_q(-\lambda)\prod_{s=0}^{j-1} \left[1+\lambda (1-q)q^{s-j}\right]^{-1} \\
&=\lambda e_q(-\lambda)\prod_{s=0}^{j-1} q^{j-s}\prod_{s=0}^{j-1} \left[q^{j-s}+\lambda (1-q)\right]^{-1} 
=:\lambda e_q(-\lambda) q^{j(j+1)/2} A_j
\end{align*}
where
\begin{align}\label{aj}
A_j=\prod_{s=1}^{j} \frac{1}{q^s+\lambda (1-q)}.
\end{align}
For $\lambda \leqslant 1/(1-q),$ one has
\begin{align*}
A_j \geqslant \prod_{s=1}^{j} \frac{1}{1+q^s}
\geqslant \prod_{s=1}^{\infty} \frac{1}{1+q^s}:=C_q.
\end{align*}
According to Theorem \ref{th1}, one concludes that the $q$-exponential distribution is $q$-moment indeterminate whenever $\lambda \leqslant 1/(1-q)$. To find a $q$-perturbation for $f$ in this case, one plugs $f(q^{-j})=\lambda/E_q(\lambda q^{-j})$ into \eqref{h1}.
\end{example}

The next outcome complements Theorem \ref{th1} by providing a condition for $q$-moment determinacy.

\begin{theorem}\label{th2}
Let $f(t),$ $t>0$ be a $q$-density of a random variable $X.$ If
\begin{align*}
f(q^{-j}) = o(q^{j(j+1)/2}) \quad \text{as} \quad j \to +\infty,
\end{align*}
then the distribution $P_X$ is $q$-moment determinate.
\end{theorem}
Prior to proving this theorem, we present an example to complete the analysis of the $q$-moment determinacy of the $q$-exponential distribution.

\begin{example} \label{ex2} If $\lambda > 1/(1-q),$ then the $q$-exponential distribution is moment determinate. Indeed, from Example \ref{ex1}, it is known that
$
f(q^{-j})=\lambda e_q(-\lambda) q^{j(j+1)/2} A_j
$
where $A_j$ is expressed by the formula \eqref{aj}.
By virtue of Theorem \ref{th2}, it suffices to show that $A_j \to 0$ as $j \to \infty.$ Since $\lambda >1/(1-q),$ it follows that
\begin{align*}
A_j=\prod_{s=1}^{j} \frac{1}{q^s+\lambda (1-q)} \leqslant \left( \frac{1}{\lambda (1-q)}\right)^j \to 0, \quad j\to\infty.
\end{align*}
\end{example}

To prove Theorem \ref{th2}, the next two auxiliary results will come in handy.

\begin{lemma}\label{lem1}
Let $\phi(z)=\sum_{j\in{\mathbb Z}} c_j z^j$ for $z\neq 0$ and $\varphi(z)=\prod_{s=1}^\infty (1-q^s z).$ If $c_j=o(q^{j(j+1)/2})$ as $j\to+\infty,$ then
$$M(r;\phi)=o(M(r;\varphi)) \quad \text{as} \quad r\to\infty.$$
\end{lemma}

\begin{proof} Let us write $$\phi(z)=\sum_{j=0}^{\infty} c_j z^j + \sum_{j=1}^\infty \frac{c_{-j}}{z^j}=:\phi_1(z)+\phi_2(z).$$
Here, $\phi_1$ is an entire function and $\phi_2$ is analytic at $\infty$ with $\phi_2(\infty)=0.$
Hence, $M(r;\phi_2) = o(1)$ as $r\to\infty.$ As for $\phi_1,$ one has
$M(r;\phi_1) \leqslant \sum_{j=0}^\infty |c_j| r^j.$
Let $\varepsilon>0$ be chosen arbitrarily. Then, there exists $j_0$ such that
$|c_j| < \varepsilon q^{j(j+1)/2}$ for $j>j_0.$ Therefore,
\begin{align*}
M(r;\phi_1) &= \sum_{j=0}^{j_0}|c_j| r^j + \sum_{j=j_0+1}^{\infty}|c_j| r^j \leqslant
P_{j_0}(r) + \varepsilon \sum_{j=j_0+1}^{\infty}  q^{j(j+1)/2}r^j \\
& \leqslant M(r;P_{j_0}) + \varepsilon \sum_{j=0}^{\infty}  \frac{q^{j(j+1)/2}}{(q;q)_j}\, r^j.
\end{align*}
Taking into account that, according to \eqref{eul},
$$
\sum_{j=0}^{\infty}  \frac{q^{j(j+1)/2}}{(q;q)_j}\, r^j = E_q\left(\frac{qr}{1-q}\right) = \varphi(-r),
$$
one derives
$$
M(r;\phi_1) \leqslant M(r;P_{j_0}) + \varepsilon \varphi(-r).
$$
As  $P_{j_0}$ is a polynomial of degree $\leq j_0,$ one has: $M(r;P_{j_0})=O\left(r^{j_0}\right),$ $r\to\infty.$  Meanwhile, $\varphi$ is a transcendental entire function, hence $\displaystyle \lim_{r\rightarrow\infty}\frac{M(r;\varphi)}{r^{j_0}}=\infty,$ implying that $M(r;P_{j_0})=o(M(r;\varphi)).$ See, for example, \cite[Chapter 1, Theorem 1]{levin}.   Using $M(r;\varphi)=\varphi(-r),$ it can be concluded that
 $$M(r;\phi)\leq M(r,P_{j_0})+o(1)\leq M(r;\phi_1)+\varepsilon M(r;\varphi)+o(1), \;r\to\infty.$$
Consequently, $M(r;\phi)\leq 2\varepsilon M(r;\varphi), r\rightarrow \infty.$ Since $\varepsilon$ was selected arbitrarily, the statement follows.
\end{proof}

\begin{lemma}\label{lem2}
Let $\phi(z)=\sum_{j\in{\mathbb Z}} c_j z^j$ satisfy $\phi(q^{-m})=0$ for all $m\in {\mathbb N}.$ Then, for $r=q^{-m},$ one has
\begin{align*}
M(r;\phi) \geqslant C \exp\left\{\frac{\ln^2r}{2\ln(1/q)}-\frac{\ln r}{2}\right\}.
\end{align*}
\end{lemma}

\begin{proof} The function $\phi$ is analytic in $0<|z|<\infty.$ Applying Jensen's Theorem \cite{titch} in the annulus
$\{z\in{\mathbb C}: 1\leqslant |z| \leqslant q^{-m}\},$ one can write, for $r=q^{-m},$
$$
\int_1^{r} \frac{n(t;\phi)}{t} \, dt \leqslant \ln M(r;\phi)+C_1
$$
where $n(t;\phi)$ is the number of zeros of $\phi$ in $1\leqslant |z| \leqslant t$ counting multiplicities. Since, $\phi$ has zeroes at $q^{-1},q^{-2},\ldots,q^{-m},$
$$
\int_1^{r} \frac{n(t;\phi)}{t} \, dt \geqslant \frac{m(m-1)}{2}\ln\left(\frac{1}{q}\right)
$$
implying that, for $r=q^{-m},$
\begin{align} \label{lnM}
\frac{m(m-1)}{2}\ln\left(\frac{1}{q}\right) \leqslant \ln C_1 M(r;\phi).
\end{align}
As $m=\ln r /\ln(1/q),$ estimate \eqref{lnM} implies, with $C=1/C_1,$
\begin{align*}
M(r;\phi) \geqslant C \exp\left\{\frac{\ln^2 r}{2\ln (1/q)}-\frac{\ln r}{2}\right\}, \quad r=q^{-m},
\end{align*}
as stated.
\end{proof}
After these two auxiliary steps, let us prove Theorem \ref{th2}.

\begin{proof}[Proof of Theorem \ref{th2}]
Let $m_q(k;X)=m_q(k;Y)$ for all $k=0,1,\ldots,$ where $Y$ is a random variable possessing a $q$-density $g(t),$ $t>0.$
Appyling the definition \eqref{mom} of the $q$-moments, one arrives at
\begin{align}\label{u}
\sum_{j=-\infty}^{\infty} f(q^{-j}) q^{-mj}=\sum_{j=-\infty}^{\infty} g(q^{-j}) q^{-mj} \quad \text{for all} \quad m=1,2,\ldots.
\end{align}
By the existence of $q$-moments of $X$ and $Y,$ both functions
\begin{align*}
\phi_1(z):=\sum_{j=-\infty}^{\infty} f(q^{-j})z^j \quad \text{ and} \quad
\phi_2(z):=\sum_{j=-\infty}^{\infty} g(q^{-j})z^j
\end{align*}
are analytic for $|z|>0,$ and so is $\phi(z):=\phi_1(z)-\phi_2(z).$
In addition, $\phi(q^{-m})=0$ for all $m=1,2,\ldots$
Now, Lemma \ref{lem2} yields
\begin{align}\label{Mrp1}
M(r;\phi) \geqslant C \exp\left\{\frac{\ln^2 r}{2\ln (1/q)}-\frac{\ln r}{2}\right\} \quad \text{when} \quad r=q^{-m}.
\end{align}
For $\varphi(z)=\prod_{s=1}^\infty(1-q^sz),$ estimate \eqref{zeng} leads to
\begin{multline}\label{Mrp2}
M(r;\varphi) = \prod_{s=1}^{\infty} (1+q^sr) = E_q\left(\frac{qr}{1-q}\right) \\
\leqslant C_2 \exp\left\{\frac{\ln^2 (qr)}{2\ln (1/q)}+\frac{\ln (qr)}{2}\right\}
\leqslant C_2 \exp\left\{\frac{\ln^2 r}{2\ln (1/q)}-\frac{\ln r}{2}\right\}.
\end{multline}
Combining \eqref{Mrp1} and \eqref{Mrp2}, one arrives at
\begin{align}\label{Mrp}
M(r;\phi) \geqslant C_3 M(r;\varphi) \quad \text{when} \quad r=q^{-m}.
\end{align}
Meanwhile, by Lemma \ref{lem1},
\begin{align}\label{os}
M(r;\phi_1)=o(M(r;\varphi)), \quad r\to\infty,
\end{align}
which - along with \eqref{Mrp} - implies that
$M(r;\phi_2)\geqslant C M(r;\varphi)$ for $r=q^{-m}$ large enough. Since all the coefficients of $\phi_1$ and $\phi_2$ are nonnegative, it follows that $M(r;\phi_1)=\phi_1(r)$ and $M(r;\phi_2)=\phi_2(r).$

Consequently,
$$M(q^{-m};\phi_1)=M(q^{-m};\phi_2)$$
as condition \eqref{u} shows. Finally, taking $r=q^{-m}$ large enough, one arrives at:
$$
M(r;\phi_1) = M(r;\phi_2) \geqslant C M(r;\varphi).
$$
This, however, contradicts \eqref{os}. The theorem is proved.
\end{proof}

\section*{Acknowledgments}
The authors would like to extend their sincere gratitude to Prof. Alexandre Eremenko from Purdue University, USA for his valuable comments and to the anonymous referees whose  careful reading and precious suggestions helped to improve the paper. 

\end{document}